\documentclass{amsart}

\usepackage{amsmath}%
\usepackage{amsfonts}%
\usepackage{amssymb}%
\usepackage{amsthm}
\usepackage{graphicx}
\usepackage{dsfont}
\usepackage{pdfcomment}
\usepackage{tikz}
\usepackage{tikz-3dplot}
\usepackage{subfigure}
\usepackage{soul}
\usepackage{enumerate}
\usepackage{cancel}

\usepackage{hyperref}
\usepackage{cleveref}
\usepackage{esvect}

\usetikzlibrary{decorations.markings}
\newtheorem{theorem}{Theorem}[section]

\newtheorem{definition}[theorem]{Definition}
\newtheorem{example}[theorem]{Example}
\allowdisplaybreaks

\newtheorem{proposition}[theorem]{Proposition}
\newtheorem{remark}[theorem]{Remark}

\DeclareMathOperator{\Ad}{Ad}

\DeclareMathOperator{\tr}{tr}

\DeclareMathOperator{\Teich}{Teich}

\title{Discrete harmonic maps between hyperbolic surfaces}
\author{Wai Yeung Lam}
\thanks{This work was partially supported by the FNR grant CoSH O20/14766753.}
\date{\today}

\address{Department of Mathematics, University of Luxembourg, Maison du nombre, 6 avenue de la Fonte, L-4364 Esch-sur-Alzette, Luxembourg.} \email{wyeunglam@gmail.com}

\begin{document}
	
		\begin{abstract}		
		Given a topological cell decomposition of a closed surface equipped with edge weights, we consider the Dirichlet energy of any geodesic realization of the 1-skeleton graph to a hyperbolic surface. By minimizing the energy over all possible hyperbolic structures and over all realizations within a fixed homotopy class, one obtains a discrete harmonic map into an optimal hyperbolic surface. We characterize the extremum by showing that at the optimal hyperbolic structure, the discrete harmonic map and the edge weights are induced from a weighted Delaunay decomposition. 
	\end{abstract}

	\maketitle

	\section{Introduction}
	
	We consider closed orientable surfaces $S$ of genus $g>1$. These surfaces support various conformal structures. By the uniformization theorem, each conformal structure can be represented by a unique hyperbolic structure. Given two different hyperbolic structures $h$ and $\tilde{h}$ on a surface, it is interesting to consider a harmonic map $f:(S,h) \to (S,\tilde{h})$, which is a homeomorphism that minimizes the Dirichlet energy
	\[
	D_h(f) = \iint ||df||^2_{\tilde{h}} \, dA_{h}
	\]
	among all immersions that are isotopic to the identity. It is known that for any fixed hyperbolic metric $(S,\tilde{h})$ at the target, the harmonic map exists uniquely.

	On can further consider the Dirichlet energy of harmonic maps by varying the targeted hyperbolic metric $\tilde{h}$ and hence regard $D_h$ as a function over the Teichm\"{u}ller space consisting of the marked hyperbolic metrics. It is a classical result that the unique minimizer of the energy takes place when the targeted hyperbolic metric coincides with the one at the source $\tilde{h}=h$ and the harmonic map becomes the identity map. In this way, one could interpret that the hyperbolic metric at the source $h$ coincides with the unique minimizer of the energy function $D_h$ over the Teichm\"{u}ller space. 
	
	On the other hand, structures on surfaces are often described in terms of combinatorial data and it is interesting to find homeomorphisms of surfaces that preserve the combinatorial structures as well. Here the combinatorial data is a cell decomposition $(V,E,F)$ of a topological surface $S$, where respectively $V,E$ and $F$ denote the set of vertices, edges and faces. Analogous to the classical harmonic map, given a hyperbolic surface $(S,\tilde{h})$, we look for an embedding of the 1-skeleton graph $(V,E)$ to the surface that minimizes the distortion energy. Fixing an arbitrary choice of positive edge weights $c:E \to \mathbb{R}_{>0}$, the Dirichlet energy is defined as
	\[
	D_c(f) = \frac{1}{2}\sum_{ij} c_{ij} \ell_{ij}^2
	\]  
	for any geodesic realization $f:(V,E) \to (S,\tilde{h})$, where edges $ij$ are realized as geodesics with hyperbolic length $\ell_{ij}$. Colin de Verdière \cite{CdV1991} proved that for a fixed hyperbolic metric, there exists a unique minimizer of the energy in the given homotopy class, which is called a discrete harmonic map. He showed that the 1-skeleton graph defines a geodesic cell decomposition of the hyperbolic surface. 
	One can further consider the Dirichlet energy of discrete harmonic maps by varying the targeted hyperbolic metric and regard $D_c$ as a function over the Teichm\"{u}ller space. Kajigaya and Tanaka \cite{Toru2021} proved that there is a unique minimizer over the Teichm\"{u}ller space, which depends on the choice of edge weights $c$. It remains mysterious how the optimal hyperbolic metric relates to the edge weights and few explicit examples of the minimizers are known.
	
	We show that the extremum corresponds exactly to a weighted Delaunay decomposition (See Definition \ref{def:wdelaunay}). Analogous results for Euclidean tori (genus g=1) was proved in \cite{Lam2022}.
	
	\begin{theorem}\label{thm:delaunay}
		Suppose $\mathcal{T}=(V,E,F)$ is a cell decomposition of a topological surface $S$ of genus larger than $1$ equipped with positive edge weights $c:E \to \mathbb{R}_{>0}$. Then at the optimal hyperbolic structure minimizing the Dirichlet energy $D_c$ over the Teichm\"{u}ller space, the discrete harmonic map and the edge weights are induced from a weighted Delaunay decomposition. 
		
		Conversely, every weighted Delaunay decomposition is the discrete harmonic map at the optimum with respect to the induced edge weights.
	\end{theorem}
 
    Given a Delaunay decomposition of a hyperbolic surface $(S,h,\mathcal{T})$, it was unclear how to pick a choice of edge weights adapting to the geometry. With the above result, we have a canonical choice (see Remark \ref{rmk:deweight} for notations)
    \begin{align}\label{eq:canon}
    	  c_{ij}:=\left(\cot( \frac{\pi - \alpha_{jk}^i - \alpha_{ki}^j + \alpha_{ij}^k}{2}) + \cot( \frac{\pi - \alpha_{lj}^i - \alpha_{il}^j + \alpha_{ji}^l}{2})\right)\frac{\tanh\frac{\ell_{ij}}{2}}{\ell_{ij}} 
    \end{align}
    which has the property that
    \begin{enumerate}[(i)]
    	\item the Dirichlet energy $D_c$ of discrete harmonic maps attains the unique minimizer over the Teichm\"{u}ller space when the targeted hyperbolic metric $\tilde{h}$ coincides with the one at the source, i.e. $\tilde{h}=h$. In this case, the discrete harmonic map is the identity map.
    		\item the discrete harmonic map $f:(V,E)\to (S,\tilde{h})$ now is not only defined topologically on the 1-skeleton graph, but can be extended to faces canonically using the hyperbolic metric $h$ at the source, which yields a map $f: (S,h,\mathcal{T}) \to (S,\tilde{h})$;
    \end{enumerate}
    The edge weight in Equation \eqref{eq:canon} also showed up in the study of discrete hyperbolic Laplacian, which is related to discrete conformal deformations in the sense of Thurston's circle patterns \cite{Lam2024,Leibon2002}. Theorem \ref{thm:delaunay} implies that we have a bijection between weighted Delaunay decompositions and positive edge weights. Thus, the associated discrete Laplacian captures the geometric information of the surface and it would be interesting to study how the geometry is related to the spectrum of the graph Laplacian.
    
     Theorem \ref{thm:delaunay} holds generally for Dirichlet energy in the forms
    \[
    \tilde{D}(\ell) = \sum_{ij} w_{ij}(\ell_{ij})
    \]
    where for each edge $ij\in E$, $w_{ij}:\mathbb{R} \to \mathbb{R}$ is a smooth increasing function (See Section \ref{sec:genen}). The discrete harmonic maps under consideration play a fundamental case in the study of generalized harmonic maps from a simplicial complex to metric spaces of non-positive curvature \cite{Wang2000}.
    
  The proof of the theorem involves the symplectic gradient of $D_c$ over the Teichm\"{u}ller space with respect to the Weil-Petersson symplectic form. It is represented as an element in the group cohomology $H^1_{\Ad \rho}(\pi_1(S),so(2,1))$, following the work of Mess \cite{Andersson2007,Mess2007} that relates the Teichm\"{u}ller space to Minkowski space.
    
    \section{Minkowski space and the Teichm\"{u}ller space}\label{sec:ttspace}
    
    We identify the Teichm\"{u}ller space $\Teich(S)$ as the space of marked hyperbolic metrics on $S$. Every hyperbolic surface is obtained via a quotient of the hyperbolic plane $\mathbb{H}^2$ by a discrete faithful representation of the fundamental group into the isometry group of the hyperbolic plane $\text{Iso}(\mathbb{H}^2)$. In order to study over the Teichm\"{u}ller space, we consider the hyperboloid model of the hyperbolic plane in Minkowski space \cite{Andersson2007,Mess2007}.
    
    \subsection{Hyperboloid in Minkowski space} The Minkowski space $\mathbb{R}^{2,1}$ is a three-dimensional real vector space equipped with the Minkowski inner product
    \[
    \langle (x_1,x_2,x_3) ,(y_1,y_2,y_3)  \rangle =x_1y_1+x_2y_2-x_3 y_3
    \]
    for $(x_1,x_2,x_3) ,(y_1,y_2,y_3)  \in \mathbb{R}^{2,1}$. A non-zero vector $x \in \mathbb{R}^{2,1}$ is space-like if $\langle x, x \rangle >0$, light-like if $\langle x, x \rangle=0$ and time-like if $\langle x, x \rangle <0$. If $x$ is space-like, we write the norm
    \[
    \| x \| :=\sqrt{\langle x, x \rangle}.
    \]
    The Minkowski cross product is defined as
    \begin{align*}
    	(x_1,x_2,x_3)  \times  (y_1,y_2,y_3)  :=
    	(x_2 y_3 -y_2 x_3) {\bf e_1} + (x_3 y_1 - x_1 y_3) {\bf e_2} -(x_1 y_2- x_2 y_1) {\bf e_0} ,
    \end{align*}
    where ${\bf e_1}:=(1,0,0) $,  ${\bf e_2}:=(0,1,0) $, ${\bf e_0}:=(0,0,1) $. For any two vectors $x, y \in \mathbb{R}^{2,1}$, $x \times y$ is perpendicular to both $x$ and $y$ with respect to the Minkowski inner product. We consider the hyperboloid model of the hyperbolic plane
    \[
    \mathbb{H}^2:=\{ x=(x_1,x_2,x_3) \in \mathbb{R}^{2,1} |  x_3>0 \text{ and } \langle x, x \rangle =-1\}.
    \] 
    The group of orientation-preserving isometries $\text{Iso}(\mathbb{H}^2)$ are identified as $SO^+(2,1)$, which is the connected component of $SO(2,1)$ containing the identity.
    
    \subsection{Tangent space of the Teichm\"{u}ller space as group cohomology}\label{sec:tanteich} Given a closed hyperbolic surface $(S,h)$, there is a developing map of the universal cover to the hyperboloid $f:\tilde{S} \to \mathbb{H}^2 \subset \mathbb{R}^{2,1}$ and it defines a holonomy representation of the fundamental group
    \[
    \rho \in \mbox{Hom}(\pi_1 (S), SO^{+}(2,1))
    \]
    satisfying for $\gamma \in \pi_1(S)$
    \begin{equation}\label{eq:fequi}
    	    f\circ \gamma = \rho_{\gamma}(f).
    \end{equation}    
    The holonomy representation is unique up to conjugation and particularly satisfies for $\gamma_1,\gamma_2 \in \pi_1(S)$ 
    \begin{equation}\label{eq:holrep}
    	  \rho_{\gamma_1 \gamma_2} = \rho_{\gamma_1} \rho_{\gamma_2}.
    \end{equation}
    Writing $ \mbox{Hom}^{*}(\pi_1 (S), SO^{+}(2,1))$ the set of discrete faithful representations, it is a classical result that
    \[
    \Teich(S) \cong \mbox{Hom}^{*}(\pi_1 (S), SO^{+}(2,1))/ \sim
    \]
    where $\rho \sim \tilde{\rho}$ if they are conjugate, i.e. there exists a constant $g\in SO^{+}(2,1)$ such that for $\gamma \in \pi_1$
    \[
    \tilde{\rho}_{\gamma} = g \rho_{\gamma} g^{-1}.
    \]
    In this way, by a result of Goldman \cite{Goldman1984}, the tangent space of the Teichm\"{u}ller space is identified with the group cohomology $H^1_{\Ad\rho}(\pi_1(S),so(2,1))$. Precisely, when one has a 1-parameter family of hyperbolic metrics with holonomy $\rho^{(t)}$ satisfying $ \rho= \rho^{(t)}|_{t=0}$ and  $\dot{\rho}:= \frac{d}{dt}\rho^{(t)} |_{t=0}$, Equation \eqref{eq:holrep} implies that the mapping \[ \sigma:= \dot{\rho} \rho^{-1} : \pi_1(S) \to so(2,1)\]  satisfies a \textit{cocycle condition}: for every $\gamma_1,\gamma_2 \in \pi_1(S)$
    \begin{align}\label{eq:cocycle}
    	 \sigma_{\gamma_1 \gamma_2}=  \sigma_{\gamma_1} + \Ad \rho_{\gamma_1} (\sigma_{\gamma_2})
    \end{align}
    where 
    \[
     \Ad \rho_{\gamma_1} (\sigma_{\gamma_2})  :=  \rho_{\gamma_1} \sigma_{\gamma_2} \rho^{-1}_{\gamma_1}.
    \]
    We write $Z^{1}_{\Ad(\rho)}(\pi_1(S),so(2,1))$ the space of cocycles, which are functions $\sigma: \pi_1(S) \to so(2,1)$ satisfying the cocycle condition (Equation \eqref{eq:cocycle}). This space contains a subspace $B^{1}_{\Ad(\rho)}(\pi_1(S),so(2,1))$ of \textit{coboundaries}. A coboundary is a function $\sigma: \pi_1(S) \to so(2,1)$ such that there exists a constant  $\sigma_0 \in so(2,1)$ satisfying for $\gamma \in \pi_1(S)$
    \[
    \sigma_\gamma =\sigma_0 - \Ad \rho_{\gamma} (\sigma_0). 
    \]
   It implies that $\dim B^{1}_{\Ad(\rho)}(\pi_1(S),so(2,1))=3$. The coboundaries correspond to the trivial change of holonomy induced from conjugation. Indeed, suppose $g^{(t)} \in SO(2,1)$ is a path in $SO(2,1)$ with $\mbox{Id}= g^{(t)}|_{t=0}$ and $\sigma_0= \frac{d}{dt} g^{(t)}|_{t=0}$. By conjugation with $g^{(t)}$, we obtain a 1-parameter family of holonomy $\rho^{(t)}$ yielding for $\gamma \in \pi_1$
   \[
   \sigma_{\gamma}= \left(\frac{d}{dt} \rho^{(t)}_{\gamma}|_{t=0}\right) \rho^{-1}_{\gamma} = \sigma_0 - \Ad \rho_{\gamma} (\sigma_0)
   \] 
   and hence $\sigma:\pi_1(S) \to so(2,1)$ is a coboundary.
   
    We then have an identification of the tangent space of the Teichm\"{u}ller space with the group cohomology defined as the quotient space
    \[ 
       T_{h}\Teich(S)\cong H^1_{\Ad\rho}(\pi_1(S),so(2,1)) := \frac{Z^{1}_{\Ad\rho}(\pi_1(S),so(2,1))}{B^{1}_{\Ad \rho}(\pi_1(S),so(2,1))}.
    \]
    \subsection{Identification between $so(2,1)$ with $\mathbb{R}^{2,1}$} \label{Section:so21R}
    
    The action of the lie algebra $so(2,1)$ on Minkowski space can be represented by the Minkowski cross product with an element in $\mathbb{R}^{2,1}$ via an isomorphism
    \[
    \eta: so(2,1) \to \mathbb{R}^{2,1}
    \]
    satisfying for any $a\in so(2,1), x\in \mathbb{R}^{2,1}$
    \[
    a(x) = \eta(a) \times x.
    \]
    Observe that for any $g \in SO(2,1)$, we have
    \[
    \eta(gag^{-1}) \times x = gag^{-1}(x) = g( \eta(a) \times g^{-1}x) = (g (\eta(a))) \times x.
    \] 
   Thus under the isomorphism $\eta$, the adjoint action of $SO(2,1)$ on $so(2,1)$ becomes
    \begin{equation}\label{eq:adeta}
 \eta(\Ad g (a)) =g(\eta(a)) 
    \end{equation}
    for any element $g \in SO(2,1)$. Recall that the Killing form on $so(2,1)$ is given via the trace operator $\tr$. The map $\eta$ is an isometry in the sense that for any $a,\tilde{a} \in so(2,1)$
    \[
    \tr(a \tilde{a}) = \langle \eta(a), \eta(\tilde{a}) \rangle.
    \]
    
    \subsection{Weil–Petersson symplectic form on the Teichm\"{u}ller space}\label{sec:wg}
    
    We express the Weil-Petersson symplectic form on the Teichm\"{u}ller space in terms of the wedge product in de Rham cohomology, following the work of Goldman \cite{Goldman1984}. We shall see that similar formulation appears to discrete harmonic maps as well. We refer \cite{FS2020, Labourie2013} for more detailed discussion.
    
    Given a closed hyperbolic surface $(S,h)$, we construct a flat vector bundle $F_{\rho}$. 
    \[
    F_{\rho} := (\mathbb{H}^2 \times so(2,1))/ \pi_1(S)\] 
    where $\pi_1(S)$ acts on $ (\mathbb{H}^2 \times  so(2,1))$ by
    \[
    \gamma \cdot (p,x) = (\rho_{\gamma}(p), \Ad \rho_{\gamma}(x))
    \]
    for $\gamma \in \pi_1(S)$. A lift of a $F_{\rho}$-valued 1-form $\alpha$ on $S$ to the universal cover $\tilde{S} \cong \mathbb{H}$ is a $so(2,1)$-valued 1-form $\tilde{\alpha}$ satisfying
    \[
    \tilde{\alpha}\circ \gamma= \Ad \rho_{\gamma} ( \tilde{\alpha}).
    \]
    We write $Z^1(S,F_{\rho})$ the space of closed $F_{\rho}$-valued 1-forms. It contains the subspace $B^1_{\text{dR}}(S,F_{\rho})$ of exact 1-forms. We then denote $H^1_{\text{dR}}(S,F_{\rho})$ the de Rham cohomology on the surface $S$ with values in $F_{\rho}$ as the quotient space
    \[
    H^1_{\text{dR}}(S,F_{\rho}) := \frac{Z^1_{\text{dR}}(S,F_{\rho})}{B^1_{\text{dR}}(S,F_{\rho})}.
    \]
     It is known that there is an isomorphism 
    \begin{align*}
    	\Phi:H^1_{\Ad(\rho)}(\pi_1(S),so(2,1)) &\to H^1_{\text{dR}}(S,F_{\rho})
    \end{align*}
    given via the periods along loops. To illustrate this mapping, let $p\in \tilde{S}$ be a lift of the base point for the fundamental group. Let $\alpha$ be a closed $F_{\rho}$-valued 1-form and $\tilde{\alpha}$ be its lift. Then one defines $\sigma:\pi_1(S) \to so(2,1)$ such that for $\gamma \in \pi_1(S)$
    \[
    \sigma_{\gamma}= \int_{p}^{\gamma (p)} \tilde{\alpha}
    \]
    where the line integral is path-independent since $\tilde{\alpha}$ is closed. For any $\gamma_1,\gamma_2 \in \pi_1(S)$, one has
    \[
    \sigma_{\gamma_1 \gamma_2}=  \int_{p}^{\gamma_1 (p)} \tilde{\alpha} + \int_{p}^{\gamma_2 (p)} \tilde{\alpha}\circ \gamma_1=\sigma_{\gamma_1 }+ \Ad\rho_{\gamma_1} (\sigma_{\gamma_2}).
    \]
    It yields that $\sigma$ is a cocycle. If $\alpha$ is changed by a $F_{\rho}$-valued exact 1-form or another lift of the base point is used, then $\sigma$ is changed by a coboundary. Thus the isomorphism is defined over the quotient space such that $\Phi([\sigma])=[\alpha]$. 
    
    On the other hand, since the universal cover is simply connected, one can consider a primitive function of the closed 1-form, namely $h:\tilde{S} \to so(2,1)$ defined by
    \[
    h(x) := \int_p^x \tilde{\alpha}.
    \]
    Observe that $h(p)=0$ and
    \begin{align*}
    	  \sigma_{\gamma} = h( \gamma(p)) - \Ad \rho_{\gamma} (h(p))=& h( \gamma(x)) - \Ad \rho_{\gamma} (h(x)) - \int_{p}^{x} (\tilde{\alpha}\circ \gamma - \Ad \rho_{\gamma} (\tilde{\alpha})) \\
    	  =& h( \gamma(x)) - \Ad \rho_{\gamma} (h(x))
    \end{align*}
    which remains constant for all $x \in \tilde{S}$.
    
    With the isomorphism $\Phi$, results of Goldman \cite{Goldman1984} showed the Weil-Petersson symplectic form can be expressed such that for $\sigma,\tau \in H^1_{\Ad(\rho)}(\pi_1(S),so(2,1))$
    \[
    \omega_G([\sigma],[\tau])=  \iint_S \tr( \Phi(\sigma) \wedge \Phi(\tau) )
    \]
    We shall expand the integral to verify that indeed it is expressed in terms solely of $\sigma$ and $\tau$. Consider a fundamental domain $\mathcal{F}$ on the universal cover by cutting $S$ along generators $\gamma_1,\gamma_2,\dots,\gamma_{2g} \in \pi_1(S)$ such that
    \[
    \gamma_1 \circ \gamma_2 \circ \gamma_1^{-1} \circ \gamma_2^{-1} \dots \gamma_{2g-1} \circ \gamma_{2g} \circ \gamma_{2g-1}^{-1} \circ \gamma_{2g}^{-1} =1 \in \pi_1(S).
    \]
    Then $\mathcal{F}$ is a 4g-polygon with sides matched in distinct pairs, i.e. the boundary is written as
    \[
    \partial \mathcal{F} = \tilde{\gamma}_1 + \tilde{\gamma}_2 + \tilde{\gamma}'_1 + \tilde{\gamma}'_2 + \dots \tilde{\gamma}_{2g-1} + \tilde{\gamma}_{2g} + \tilde{\gamma}'_{2g-1} + \tilde{\gamma}'_{2g}
    \]
    where $ \tilde{\gamma}_r$ and  $\tilde{\gamma}'_r$ are some lifts of the loops $\gamma_r$ and $\gamma_r^{-1}$ to the universal cover. For $r=1,2,\dots,2g$, there is a unique $\delta_r \in  \pi_1 (S)$ carrying a side $ \tilde{\gamma}_r$ to the paired side $ \tilde{\gamma}_r'$ reversing the orientation via a deck transformation. We write $\alpha$, $\beta$ be the lift of closed 1-forms representing $\Phi(\sigma)$ and $\Phi(\tau)$. Let $h:\tilde{S} \to so(2,1)$ be a primitive function of $\alpha$. With these, we expand the integral using Stokes' theorem
    \begin{align}\label{eq:wgexplicit}
    	\begin{split}
    			\omega_G(\sigma, \tau)  &=  \int_{\partial F_g}  \tr (h \beta) \\
    		&= \sum_{r=1}^{2g}\left( \int_{\gamma_r} \tr( f \beta) -  \int_{\gamma_r} \tr( f\circ \delta_r \cdot \beta\circ \delta_r) \right) \\
    		&= \sum_{r=1}^{2g}\left( \int_{\gamma_r} \tr( \Ad \rho_{\delta_r} (f) \cdot \Ad \rho_{\delta_r} (\beta)) -  \int_{\gamma_r} \tr( f\circ \delta_r \cdot \beta\circ \delta_r) \right) \\ 
    		&= \sum_{r=1}^{2g}\left( \int_{\gamma_r} \tr( \Ad \rho_{\delta_r} (f) \cdot \beta \circ \delta_r) -  \int_{\gamma_r} \tr( f\circ \delta_r \cdot \beta\circ \delta_r) \right) \\
    		&= -\sum_{r=1}^{2g}\tr\left( (f\circ \delta_r - \Ad \rho_{\delta_r} (f))    (\int_{\gamma_r} \beta \circ \delta_r) \right) \\
    		&= -\sum_{r=1}^{2g}\tr\left( \sigma_{\delta_r}  (\int_{\gamma_r} \beta \circ \delta_r) \right)
    	\end{split}
    \end{align}
    where we used the fact for each $r$,  the function $(f\circ \delta_r - \Ad \rho_{\delta_r} (f))$ is constant over the universal cover $\tilde{S}$ and equal to $\sigma_{\delta_r} $. The line integral 
    \[
    \int_{\gamma_r} \beta \circ \delta_r
    \]
    is the evaluation of $\tau$ at some element in $\pi_1(S)$. We will see that the same expression appears to discrete harmonic maps.
    
      	\section{Weighted Delaunay decomposition}\label{sec:delaunay}
	
	\subsection{Delaunay decomposition}
	Given a topological cell decomposition $(V,E,F)$ of $S$, a \textit{geodesic decomposition} of a hyperbolic surface $(S,h)$ is an embedding of the 1-skeleton graph such that every edge is realized as a piece of non-degenerate geodesic. By lifting to the universal cover, every geodesic decomposition of a closed hyperbolic surface $(S,h)$ induces a geodesic decomposition of the hyperbolic plane. It yields a developing map $f:(\tilde{V},\tilde{E}) \to \mathbb{H}^2\subset \mathbb{R}^{2,1}$ equivariant with respect to $\rho$, where $\rho$ is the holonomy associated to the hyperbolic structure $h$. Particularly, for any $\gamma \in \pi_1 (S)$ and $i \in \tilde{V}$
	\[
	f_{\gamma(i)}=\rho_{\gamma} (f_i).
	\]
	
	A \textit{Delaunay decomposition} is a geodesic decomposition satisfying an empty disk condition: for each face, there is a circle passing through all vertices of the face and does not contain any other vertices in the associated closed disk. Recall that a hyperbolic circle on $\mathbb{H}^2 \subset \mathbb{R}^{2,1}$ is the intersection of an affine plane with the hyperboloid $\mathbb{H}^2$. It yields that for a Delaunay decomposition, the developing map $f:\tilde{V} \to \mathbb{H}^2\subset \mathbb{R}^{2,1}$ has planar faces and for each face, all other vertices lie above the associated affine plane. Thus, it defines a $\rho$-equivariant convex polyhedral surface in $\mathbb{R}^{2,1}$. The converse statement holds as well. Namely, a geodesic decomposition is Delaunay if and only if the developing map $f$ yields a convex polyhedral surface with the same combinatoric.
	
	Given the convex polyhedral surface $f$ coming from a Delaunay decomposition, we consider a dual convex polyhedral surface $f^{\dagger}$ in $\mathbb{R}^{2,1}$, which is the polar of $f$ with respect to the hyperboloid. Combinatorially, $f$, $f^{\dagger}$ are related such that vertices of one correspond to the faces of the other while edges joining pairs of vertices of one correspond to edges shared between pairs of faces of the other. Geometrically, for each $i\in \tilde{V}$, the tangent plane of the hyperboloid at $f_i$ contains the corresponding dual face under $f^{\dagger}$. For each $\phi \in \tilde{F}$, the point $f^{\dagger}_{\phi}$ is the intersection of the tangent planes at $f_i$'s, where $i$ is a vertex of $\phi$. Particularly,
	\begin{align}\label{eq:polar}
		\langle f^{\dagger}_{\phi}, f_i \rangle =-1.
	\end{align}
	It yields that the dual polyhedral surface $f^{\dagger}$ has space-like faces and $\rho$-equivariant:  for any $\gamma \in \pi_1 (S)$ and $\phi \in \tilde{F}$
	\[
	f^{\dagger}_{\gamma(\phi)}=\rho_{\gamma} (f^{\dagger}_{\phi}).
	\]
	
	We consider a canonical edge weight $c:E \to \mathbb{R}_{>0}$ associated to a Delaunay decomposition
	\begin{equation}\label{eq:cedgeweight}
	c_{ij}:=	\frac{|| f^{\dagger}_{ij,l} - f^{\dagger}_{ij,r}||}{\ell_{ij}} 
	\end{equation}
	where $(ij,l)$ and $(ij,r)$ denote the left face and the right face of the oriented edge $ij$ while $\ell_{ij}$ is the hyperbolic length of the geodesic edge $ij$. Observe that following Equation \eqref{eq:polar}, we thus have for every oriented edge $ij$,
	\begin{align*}
     f^{\dagger}_{ij,l}-f^{\dagger}_{ij,r}= \frac{||f^{\dagger}_{ij,l}-f^{\dagger}_{ij,r}||}{||f_i \times f_j||}(f_i \times f_j) = \frac{ c_{ij} \ell_{ij}}{\sinh \ell_{ij}} (f_i \times f_j).
	\end{align*}

    The duality between polyhedral surfaces $f$ and $f^{\dagger}$ will be used throughout the following sections. It can be seen as the duality between Delaunay decompositions and Voronoi diagrams, which we shall not explore here.
	
	\begin{remark}\label{rmk:deweight}
		The induced edge weight associated to a Delaunay triangulation in Equation \eqref{eq:cedgeweight} has an explicit formula. As computed in \cite{Lam2024}, with notations in Figure \ref{fig:triangle}, we have
		\[
			|| f^{\dagger}_{ij,l} - f^{\dagger}_{ij,r}||= \left(\cot( \frac{\pi - \alpha_{jk}^i - \alpha_{ki}^j + \alpha_{ij}^k}{2}) + \cot( \frac{\pi - \alpha_{lj}^i - \alpha_{il}^j + \alpha_{ji}^l}{2})\right) \tanh\frac{\ell_{ij}}{2}
		\]
		and thus
		\[
		c_{ij} =  \left(\cot( \frac{\pi - \alpha_{jk}^i - \alpha_{ki}^j + \alpha_{ij}^k}{2}) + \cot( \frac{\pi - \alpha_{lj}^i - \alpha_{il}^j + \alpha_{ji}^l}{2})\right) \frac{\tanh\frac{\ell_{ij}}{2}}{\ell_{ij}}.
		\]
		\begin{figure}
			\centering
		\includegraphics[width=0.8\textwidth]{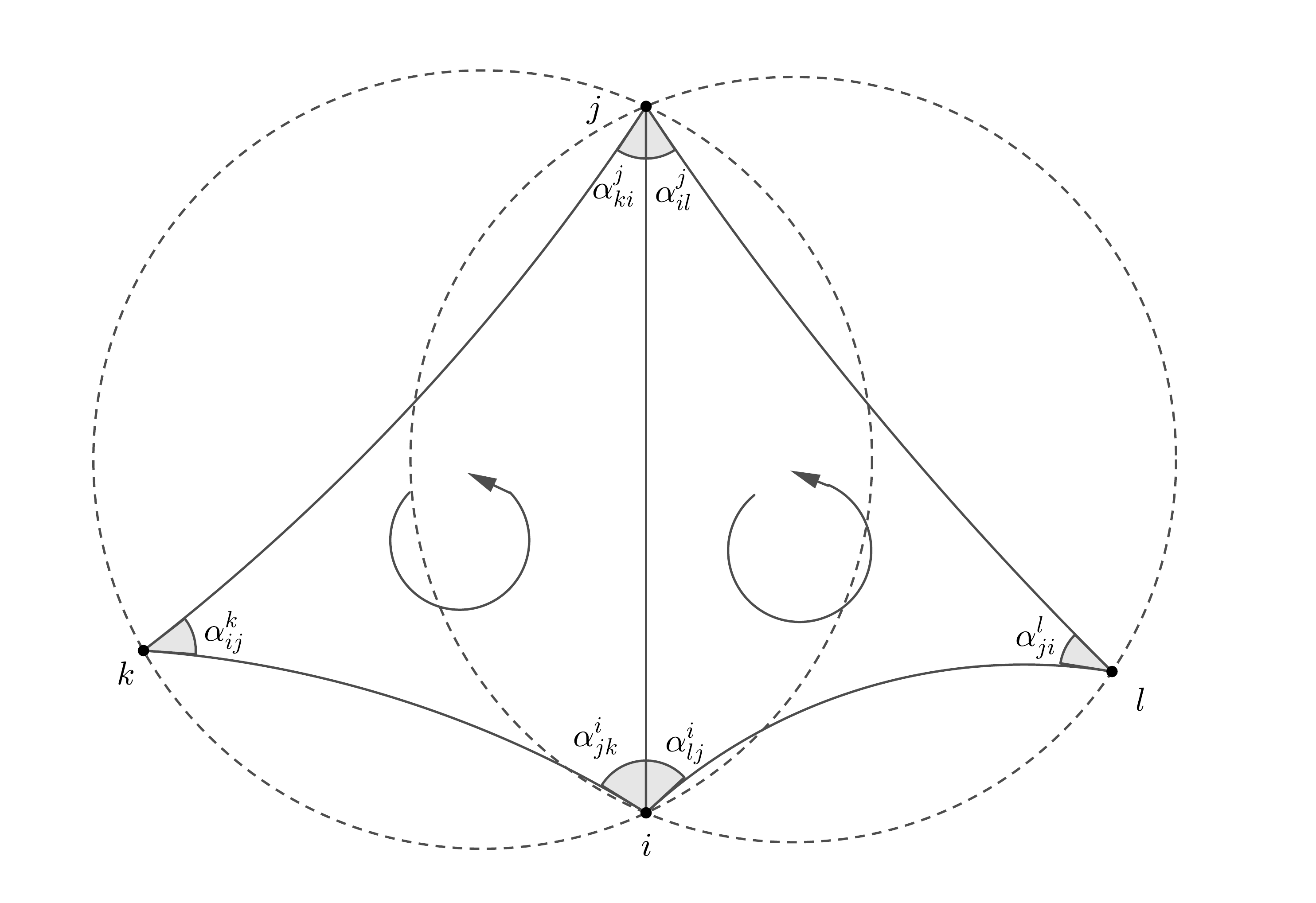}
		\caption{Two neighbouring hyperbolic triangles in a Delaunay triangulation.}
		\label{fig:triangle}
	\end{figure}
	\end{remark}

	\subsection{Weighted Delaunay decomposition}
	
	For the sake of the main arguments, we take an abstract definition of a weighted Delaunay decomposition.
	
	\begin{definition}\label{def:wdelaunay}
		A geodesic decomposition of $(S,h)$ is a weighted Delaunay decomposition with vertex weight $\delta:V \to \mathbb{R}_{>0}$ if there is a $\rho$-equivariant dual convex polyhedral surface $f^{\dagger}$ such that for every face $\phi \in \tilde{F}$ and any $i \in \phi$,
		\begin{align}\label{eq:polard}
				\langle f^{\dagger}_{\phi}, f_i \rangle = -\delta_{p(i)}
		\end{align}
		where $\rho$ is the holonomy associated to the hyperbolic metric and $f:(\tilde{V},\tilde{E}) \to \mathbb{H}^2 \subset \mathbb{R}^{2,1}$ is the developing map. Here $p$ is the projection of the universal cover to the closed surface. The induced edge weight is defined by Equation \eqref{eq:cedgeweight}.
	\end{definition}
	
	A Delaunay decomposition is a weighted Delaunay decomposition with vertex weight $\delta \equiv 1$. The dual surface $f^{\dagger}$ is uniquely determined by $f$ and $\delta$ via Equation \eqref{eq:polard} if exists. There are geodesic decompositions that are not weighted Delaunay for any choice of vertex weights (See \cite[Example 6.4]{Lam2022}). 
	
	Rewriting Equation \eqref{eq:polard} into
	\begin{align*}
		\langle f^{\dagger}_{\phi}, \tilde{f}_i \rangle = -1, \quad \text{where}\quad  \tilde{f}_i:= \frac{f_i}{\delta_i}
	\end{align*} 
   yields that $\tilde{f}$ defines a convex polyhedra surface polar dual to $f^{\dagger}$ with respect to the hyperboloid. Both surfaces $\tilde{f}$ and $f^{\dagger}$ are $\rho$-equivariant. Equation \eqref{eq:polard} also implies that for every oriented edge $ij$,
   \begin{align*}
   	f^{\dagger}_{ij,l}-f^{\dagger}_{ij,r}= \frac{||f^{\dagger}_{ij,l}-f^{\dagger}_{ij,r}||}{||\tilde{f}_i \times \tilde{f}_j||}(\tilde{f}_i \times \tilde{f}_j) = \frac{ c_{ij} \ell_{ij}}{\sinh \ell_{ij}} (f_i \times f_j)
   \end{align*}
    where $(ij,l)$ and $(ij,r)$ denote the left face and the right face of the oriented edge $ij$ while $\ell_{ij}$ is the hyperbolic length of the geodesic edge $ij$.
	
	Here we explain briefly the geometric meaning of the vertex weight when $\delta$ is close to $1$ (See \cite{Springborn2008} for further details).  The polar of $\tilde{f}_i$ is an affine plane containing the corresponding dual face under $f^{\dagger}$. When $\delta_i\geq1$, this affine plane intersects with the hyperboloid and yields a hyperbolic circle centered at $f_i$ with radius $\cosh^{-1} \delta_i$. This associate every vertex with a circle well-defined on the closed hyperbolic surface. When $\delta$ is furthermore close to $1$, the vertex circles are disjoint. Instead of empty disk condition, the Delaunay condition means that there is a face circle 
	intersecting its associated vertex circles orthogonally but no other vertex circles more than orthogonally. Indeed, the face circles are encoded by $f^{\dagger}$. The face circle corresponding to $\phi \in \tilde{F}$ is centered at $-\frac{f^{\dagger}_{\phi}}{||f^{\dagger}_{\phi}||} \in \mathbb{H}^2 \subset \mathbb{R}^{2,1}$ with radius $\cosh^{-1}(-||f^{\dagger}_{\phi}||)$, which is the intersection of the hyperboloid with an affine plane polar to $f^{\dagger}_{\phi}$. 
	When $\delta_i<1$, the vertex circle becomes imaginary and needs to be interpreted differently. Anyhow, the weighted Delaunay decomposition can be defined using the Voronoi diagram with distance modified by such vertex weights.

\section{Critical points of energy}	
	
	 The Dirichlet energy is associated to any geodesic realization of the 1-skeleton graph. We first consider critical points of the energy among all geodesic realizations within a homotopy class to a fixed hyperbolic surface. These critical geodesic realizations are called discrete harmonic maps. Afterwards, we investigate critical points of the energy where the targeted hyperbolic is allowed to vary over the Teichm\"{u}ller space, which consists of marked hyperbolic metrics.

	\subsection{Variation over a fixed hyperbolic surface}
	
	We first define a discrete harmonic map.

	 \begin{definition}
		A geodesic realization of $(V,E)$ into a hyperbolic surface $(S,\tilde{h})$ is a discrete harmonic map if it is a critical point of the energy $D_c$ among all geodesic realizations into the fixed hyperbolic surface within the same homotopy class.
	\end{definition}
	
	Colin de Verdière \cite{CdV1991} proved that for a fixed hyperbolic metric at the target, there exists a unique minimizer of the energy in the given homotopy class. In order to deduce the gradient of the energy over the Teichm\"{u}ller space, we consider several equivalent statements for discrete harmonic maps.
	
		We fix a hyperbolic metric $\tilde{h}$ and consider a geodesic realization of the 1-skeleton graph $(V,E)$ to the fixed hyperbolic surface $(S,\tilde{h})$. We lift it to the universal cover and obtain $f:(\tilde{V},\tilde{E}) \to \mathbb{H}^2 \subset \mathbb{R}^{2,1}$. For $ij \in E$, the geodesic length $\ell_{ij}$ of the edge is given via
	\[
	\cosh \ell_{ij} = -\langle f_i, f_j \rangle
	\]
	or equivalently
	\[
	\sinh \ell_{ij} = ||f_i \times f_j||.
	\]
    Suppose we have a 1-parameter family of geodesic realizations into the fixed hyperbolic surface $f^{(t)}:(\tilde{V},\tilde{E}) \to \mathbb{H}^2$ with $f= f^{(t)}|_{t=0}$ where the holonomy representation $	\rho \in \mbox{Hom}(\pi_1 M, SO(2,1))$ remains unchanged. We also write $\dot{f} = \frac{d}{dt} f^{(t)}|_{t=0}$ which satisfies for any $i \in \tilde{V}$ and $\gamma \in \pi_1(S)$
    \[
    	\dot{f}_{\gamma(i)}=\rho_{\gamma} (\dot{f}_i)
    \]
	and $\langle f_i , \dot{f}_i \rangle =0$. Since $\langle \dot{f},f \rangle =0$, we have the following identify
	\[
	\langle f_j- f_i,\dot{f}_j - \dot{f}_i \rangle =-( \langle \dot{f}_i, f_j \rangle + \langle f_i,\dot{f}_j \rangle) =   \langle f_i \times f_j,f_j \times \dot{f}_j - f_i \times \dot{f}_i \rangle.
	\]
	For every edge $ij\in E$ of the closed surface, we pick a lift of it on the universal cover. The collection of such edges form a subset $E'\subset \tilde{E}$ and there is a bijection $E'\cong E$. Abusing the notation, we also write $ij \in E'$ for the lift. Later on, we will specify $E'$ to be the edges in a fundamental domain.
	
	Now the change of Dirichlet energy can be written as
	\begin{align*}
			\frac{d}{dt} D_c(f^{(t)})|_{t=0}& = \sum_{ij\in E} c_{ij} \ell_{ij} \dot{\ell}_{ij}\\
			&=- \sum_{ij \in E'} \frac{ c_{ij} \ell_{ij}}{\sinh \ell_{ij}}( \langle \dot{f}_i, f_j \rangle + \langle f_i,\dot{f}_j \rangle) \\ &=  \sum_{ij \in E'} \frac{ c_{ij} \ell_{ij}}{\sinh \ell_{ij}} \langle f_j- f_i,\dot{f}_j - \dot{f}_i \rangle  \\
			&= \sum_{ij \in E'} \frac{ c_{ij} \ell_{ij}}{\sinh \ell_{ij}} \langle f_i \times f_j,f_j \times \dot{f}_j - f_i \times \dot{f}_i \rangle
	\end{align*}
   It yields the following observations.
   
    \begin{proposition}\label{prop:dualsurface}
    	Denote $f:(\tilde{V},\tilde{E}) \to \mathbb{H}^2 \subset \mathbb{R}^{2,1}$ the lift of a geodesic realization to the universal cover. The following characterizations are equivalent conditions for $f$ to be a discrete harmonic map:
        \begin{enumerate}[\hspace{0.2cm} (a)]
        	\item Critical point of the energy $D_c$ among all geodesic realizations into the fixed hyperbolic surface within the same homotopy class.
        	\item There exists a function $\mu:\tilde{V} \to \mathbb{R}$ such that for every fixed $i\in \tilde{V}$
        	\begin{equation}\label{eq:maxa}
        		\sum_{j} \frac{c_{ij} \ell_{ij}}{\sinh \ell_{ij}} (f_j- f_i) = \mu_i f_i
        	\end{equation}
        	where the summation is over all edges connecting to the vertex $i$ on the lift $(\tilde{V},\tilde{E})$. Here $\mu$ plays the role of the Lagrange multiplier for the constraint $\langle f, f \rangle \equiv -1$, satisfying $\mu_{\gamma(i)}= \mu_i$ for $\gamma \in \pi_1(S)$.  
        	\item For every fixed $i\in \tilde{V}$
        		\begin{equation}\label{eq:maxc}
        	\sum_{j} \frac{ c_{ij} \ell_{ij}}{\sinh \ell_{ij}} (f_i \times f_j) = 0
        \end{equation}
        	where the summation is over all edges connecting to the vertex $i$ on the universal cover.
        	\item There exists a realization of the dual graph $(V^*,E^*)$ into Minkowski space $\mathbb{R}^{2,1}$ such that
        	\[
        	f^{\dagger}_{ij,l} - f^{\dagger}_{ij,r}= \frac{ c_{ij} \ell_{ij}}{\sinh \ell_{ij}} (f_i \times f_j)
        	\]
        	for every $ij \in \tilde{E}$. The map $f^{\dagger}$ is unique up to translation in $\mathbb{R}^{2,1}$.
        \end{enumerate}
    \end{proposition}
\begin{proof}
	To obtain (a)$\leftrightarrow$(b), consider an infinitesimal deformation moving vertex $i$ only but with other vertices fixed. Then
	\[
		\frac{d}{dt} D_c(f^{(t)})|_{t=0} = 	-\langle \sum_{j} \frac{c_{ij} \ell_{ij}}{\sinh \ell_{ij}} (f_j- f_i),\dot{f}_i\rangle.
	\]
	Since $\dot{f}_i$ is an arbitrary vector in the tangent plane, we deduce the sum has to be parallel to the normal direction $f_i$ and hence Equation \eqref{eq:maxa}. The relation (b)$\leftrightarrow$(c) is obtained by taking the cross product with $f_i$ on both sides. Equation \eqref{eq:maxc} implies that the function $\mathbb{R}^{2,1}$-valued function 
	\[
	a_{ij}:=\frac{ c_{ij} \ell_{ij}}{\sinh \ell_{ij}} (f_i \times f_j)
	\]
	satisfying $a_{ij}=-a_{ji}$ is closed around every dual face. Hence it can be integrated on the universal cover. 
\end{proof}

The realization of the dual graph is crucial to our study. 
   
   \begin{proposition}\label{prop:dualsurcocycle}
   	Every discrete harmonic map to a hyperbolic surface $(S,h)$ induces a realization of the dual graph $f^{\dagger}:(V^*,E^*) \to \mathbb{R}^{2,1}$ and a cocycle $\tau: \pi_1(S) \to so(2,1)$ such that for $\gamma \in \pi_1(S)$ and $\phi \in \tilde{F}$
   	  \[
   	f^{\dagger}_{\gamma(\phi)}- \rho_\gamma(f^{\dagger}_{\phi})=\eta(\tau_{\gamma})  
   	\]
   	where $\rho \in \mbox{Hom}(\pi_1 (S), SO(2,1))$ is the holonomy representation of the hyperbolic metric $h$. Whenever $f^{\dagger}$ is translated by a constant, then $\tau$ differs by a coboundary. Thus every discrete harmonic map yields a well-defined equivalence class
   	\[
   	[\tau] \in H^1_{\Ad(\rho)}(\pi_1(S),so(2,1)) \cong T_{h}\Teich(S).
   	\]
   \end{proposition}
\begin{proof}
The previous Proposition \ref{prop:dualsurface} shows that there exists a realization of the dual graph $f^{\dagger}:(V^*,E^*) \to \mathbb{R}^{2,1}$ unique up to translation. Since $f$ is equivariant with respect to $\rho \in \mbox{Hom}(\pi_1(S), SO(2,1))$, we have by construction that for any $\gamma \in \pi_1(S)$ and $ij \in \tilde{E}$
\[
f^{\dagger}_{\gamma(ij,l)} - f^{\dagger}_{\gamma(ij,r)} =\frac{ c_{ij} \ell_{ij}}{\sinh \ell_{ij}} (f_{\gamma(i)} \times f_{\gamma(j)}) = \rho_{\gamma}(	f^{\dagger}_{ij,l} - f^{\dagger}_{ij,r}).
\]
It implies for $\gamma \in \pi_1(S)$, there exists a translation vector $\mathfrak{t}_{\gamma} \in \mathbb{R}^{2,1}$ independent of $\phi \in \tilde{F}$ such that 
\[
f^{\dagger}_{\gamma(\phi)}- \rho_\gamma(f^{\dagger}_{\phi})=\mathfrak{t}_{\gamma}.
\]
In this way, we get a mapping $\mathfrak{t}: \pi_1(S) \to \mathbb{R}^{2,1}$ satisfying for every $\gamma_1,\gamma_2 \in \pi_1(S)$ 
	\begin{align*}
		\mathfrak{t}_{\gamma_1 \gamma_2}=&f^{\dagger}_{\gamma_1\gamma_2(\phi)}- \rho_{\gamma_1\gamma_2}(f^{\dagger}_{\phi}) \\=& f^{\dagger}_{\gamma_1\gamma_2(\phi)}-\rho_{\gamma_1}(f^{\dagger}_{\gamma_2(\phi)}) + \rho_{\gamma_1}(f^{\dagger}_{\gamma_2(\phi)}) - \rho_{\gamma_1} \rho_{\gamma_2}(f^{\dagger}_{\phi})\\ =& \mathfrak{t}_{\gamma_1} + \rho_{\gamma_1} (\mathfrak{t}_{\gamma_2}).
	\end{align*}
	Hence, together with Equation \eqref{eq:adeta}, we deduce that $\tau:=\eta^{-1} \circ \mathfrak{t}$ is a cocycle and represents $[\tau] \in H^1_{\Ad(\rho)}(\pi_1(S),so(2,1))$. 
	
	Suppose $\tilde{f}^{\dagger}= f^{\dagger}+c$ for some constant vector $c \in \mathbb{R}^{2,1}$. The realization $\tilde{f}^{\dagger}$ also defines a mapping $\tilde{\mathfrak{t}}: \pi_1(S) \to \mathbb{R}^{2,1}$ and a cocycle $\tilde{\tau}=\eta^{-1} \circ \tilde{\mathfrak{t}}$. For every $\gamma \in \pi_1 (S)$, we have
	\[
	\tilde{\mathfrak{t}}_{\gamma} - \mathfrak{t}_{\gamma}= c - \rho_{\gamma}(c)
	\]
	which implies
	\[
	\tilde{\tau}_{\gamma}-\tau_{\gamma} = \eta^{-1}(c) - \Ad \rho_{\gamma} (\eta^{-1}(c))
	\]
	Hence $\tilde{\tau}-\tau$ is a coboundary and $[\tilde{\tau}]=[\tau]$.
\end{proof}
	
	In the next section, we shall show that the equivalence class $[\tau]$ is the symplectic gradient of the energy $D_c$ over the Teichm\"{u}ller space.

	\begin{remark}
		Another equivalent characterization for a discrete harmonic map is that	for every fixed $i\in \tilde{V}$
		\[
		\sum_j c_{ij} \ell_{ij} U_{ij} =0.
		\]
		where $U_{ij}$ is the unit tangent vector of the geodesic from vertex $f_i$ to $f_j$ (See \cite{Toru2021}).
	\end{remark}
	
	\begin{remark}
		In the case that the edge weight $c$ is allowed to take negative values and the Lagrange multiplier $\mu \equiv 0$, the dual surfaces defined via Equation \eqref{eq:maxa} and \eqref{eq:maxc} are characterized as discrete maximal surfaces \cite{LamYashi2017}. These are cousins of minimal surfaces in Euclidean space \cite{Lam2016}.   
	\end{remark}
	
	\subsection{Variation over the Teichm\"{u}ller space} 
	
	Given a discrete harmonic map to a closed hyperbolic surface, we denote its lift $f:(\tilde{V},\tilde{E}) \to \mathbb{H}^2 \subset \mathbb{R}^{2,1}$. We consider a 1-parameter family of geodesic realizations into closed hyperbolic surfaces. Its lift to the universal cover $f^{(t)}:(\tilde{V},\tilde{E}) \to \mathbb{H}^2\subset \mathbb{R}^{2,1}$ satisfies $f= f^{(t)}|_{t=0}$ and $\dot{f} = \frac{d}{dt} f^{(t)}|_{t=0}$. Here the hyperbolic metrics at the target are allowed to change, which defines a 1-parameter family of holonomy representation $\rho^{(t)} \in \mbox{Hom}(\pi_1(S), SO(2,1))$. Denote $\dot{\rho}:= \frac{d}{dt}\rho^{(t)} |_{t=0}$. As discussed in Section \ref{sec:tanteich}, the infinitesimal deformation defines a cocycle \[ \sigma:=\dot{\rho} \rho^{-1}: \pi_1(S) \to so(2,1)\] representing an element 
	\[
	[\sigma] \in   H^1_{\Ad(\rho)}(\pi_1(S),so(2,1)).
	\]
	\begin{proposition}With the notations above, the change of discrete Dirichlet energy of a discrete harmonic map is
		\[
			\frac{d}{dt} D_c(f^{(t)})|_{t=0} = \omega_G(\sigma,\tau)
		\]
	where $\tau$ is a cocycle induced from the dual surface of the discrete harmonic map. Therefore, $[\tau]\in  H^1_{\Ad(\rho)}(\pi_1(S),so(2,1))$ is the symplectic gradient of the Dirichlet energy over the Teichm\"{u}ller space.
	\end{proposition}
	\begin{proof}
	Equation \eqref{eq:fequi} implies that the infinitesimal deformation $\dot{f}$ satisfies for any $i \in \tilde{V}$ and $\gamma \in \pi_1(S)$ 
\[
\dot{f}_{\gamma(i)}-\rho_{\gamma} (\dot{f}_i) = \dot{\rho}_{\gamma}(f_i).
\]
By the identification $\eta:so(2,1)\to \mathbb{R}^{2,1}$ in Section \ref{Section:so21R}, we have for all $x \in \mathbb{R}^{2,1}$
\[
\dot{\rho}_{\gamma} \rho^{-1}_{\gamma} (x) = \eta(\sigma_{\gamma}) \times x
\]
where the Minkowski cross product is used. For any $\gamma \in \pi_1(S)$, we thus have
\begin{align*}
	f_{\gamma(i)} \times \dot{f}_{\gamma(i)} &= 	\rho_{\gamma} f_i \times (\rho_{\gamma} \dot{f}_i  +\eta(\sigma_{\gamma}) \times \rho_{\gamma} f_i ) \\
	&=\rho_{\gamma}(f_i \times \dot{f}_i) + \rho_{\gamma} f_i \times (\eta(\sigma_{\gamma}) \times \rho_{\gamma} f_i) \\
	&= \rho_{\gamma}(f_i \times \dot{f}_i) + \eta(\sigma_{\gamma}) + \langle \eta(\sigma_{\gamma}), f_{\gamma(i)} \rangle f_{\gamma(i)}  
\end{align*}
where the formula for the triple Minkowski cross product is used: for any $u,v \in \mathbb{R}^{2,1}$
\[
u \times (v \times u)=-\langle u, u \rangle v + \langle u,v \rangle u.
\]

In order to express the derivative of the energy in terms of the symplectic form, we consider a fundamental domain on the universal cover consisting of dual faces as in Section \ref{sec:wg}. In fact, any fundamental domain works fine but we pick a standard one in order to simplify our notations. Based at a dual vertex, we cut $S$ along generators $\gamma_1,\gamma_2,\dots,\gamma_{2g} \in \pi_1 (S)$ such that
\[
\gamma_1 \circ \gamma_2 \circ \gamma_1^{-1} \circ \gamma_2^{-1} \dots \gamma_{2g-1} \circ \gamma_{2g} \circ \gamma_{2g-1}^{-1} \circ \gamma_{2g}^{-1} =1 \in \pi_1 (S).
\]
The loops consist of edges in the dual cell decomposition. We thus obtain a fundamental domain $\mathcal{F}$ on the universal cover. It is a 4$g$-polygon with boundary
\[
\partial \mathcal{F} = \tilde{\gamma}_1 + \tilde{\gamma}_2 + \tilde{\gamma}'_1 + \tilde{\gamma}'_2 + \dots \tilde{\gamma}_{2g-1} + \tilde{\gamma}_{2g} + \tilde{\gamma}'_{2g-1} + \tilde{\gamma}'_{2g}
\]
where $ \tilde{\gamma}_r$ and  $\tilde{\gamma}'_r$ are some lifts of the loops $\gamma_r$ and $\gamma_r^{-1}$ to the universal cover. For $r=1,2,\dots,2g$, there is a unique $\delta_r \in  \pi_1 (S)$ carrying a side $ \tilde{\gamma}_r$ to the paired side $ \tilde{\gamma}_r'$ reversing the orientation via a deck transformation. We define $E' \subset \tilde{E}$ to be the collection of primal edges which are dual to edges either in the interior of $\mathcal{F}$ or edges in $\tilde{\gamma}_r$ for some $r$. Particularly, $E'$ does \textit{not} include primal edges that are dual to edges in  $\tilde{\gamma}'_r$ so that we have $E'$ being bijective with the set of edges on the closed surface. Recall that we write $(ij,l)$ and $(ij,r)$ the left face and the right face of the edge $ij$ oriented from $i$ to $j$. Thus we have
\[
f^{\dagger}_{ij,l} - f^{\dagger}_{ij,r} = -(f^{\dagger}_{ji,l} - f^{\dagger}_{ji,r}).
\]

Now we can write the change of energy as
\begin{align*}
	\frac{d}{dt} D_c(f^{(t)})|_{t=0}&=  \sum_{ij \in E'} \langle 	f^{\dagger}_{ij,l} - f^{\dagger}_{ij,r} ,f_j \times \dot{f}_j - f_i \times \dot{f}_i \rangle \\
	&= -\sum_{i\in \tilde{V}} \sum_{ ij \in E'| i} \langle 	f^{\dagger}_{ij,l} - f^{\dagger}_{ij,r} , f_i \times \dot{f}_i \rangle
\end{align*}
	\begin{figure}
	\centering
	\includegraphics[width=0.8\textwidth]{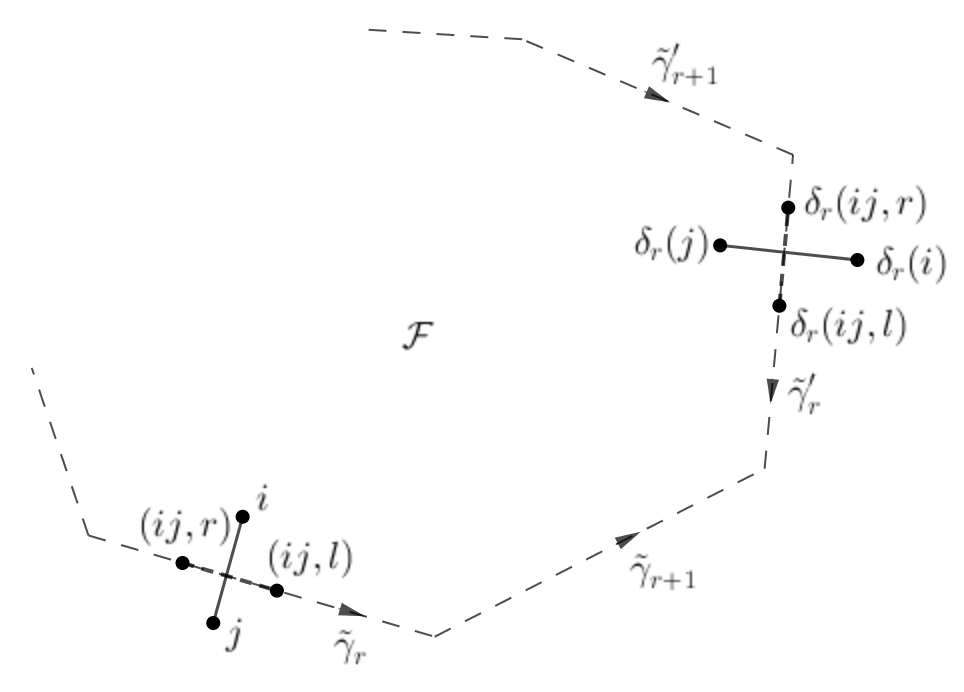}
	\caption{The fundamental domain $\mathcal{F}$ with $*ij\in \tilde{\gamma}_r \subset \partial \mathcal{F}$. Here $*ij$ is the dual edge of $ij$ and it is oriented from $(ij,r)$ to $(ij,l)$. By construction, the edge $ij$ is in $E'$. However, $\delta_r(ij)$ is not in $E'$ since its dual edge is in $\tilde{\gamma}'_r$.}
	\label{fig:fun}
\end{figure}
The summation 
\[
\sum_{ ij \in E'| i} \langle 	f^{\dagger}_{ij,l} - f^{\dagger}_{ij,r} , f_i \times \dot{f}_i \rangle
\]
is over all edges in $E'$ with $i$ as a starting point and the edges are oriented outward. It is equal to zero if $i$ corresponds to a dual face in $\mathcal{F}$ not intersecting any $\tilde{\gamma}_r'$. On the other hand, if $ij\in E'$ intersect $\tilde{\gamma}_r$ form some $r$, then one of the end point, say $j$, lies outside $\mathcal{F}$ (See Figure \ref{fig:fun})). The edge $\delta_r(ij)$ then intersect $\tilde{\gamma}_r'$ with $\delta_r(j)$ lying inside $\mathcal{F}$. We write $*ij$ the dual edge of $ij$ oriented from $(ij,r)$ to $(ij,l)$. With these observations, we have 
\begin{align*}
	&\frac{d}{dt} D_c(f^{(t)})|_{t=0}\\=&  \sum_{r=1}^{2g} \sum_{*ij \in \tilde{\gamma}_r}  \left(\langle 	f^{\dagger}_{ij,l} - f^{\dagger}_{ij,r} ,f_j \times \dot{f}_j \rangle - \langle 	f^{\dagger}_{\delta_r(ij,l)} - f^{\dagger}_{\delta_r(ij,r)}, f_{\delta_r(j)} \times \dot{f}_{\delta_r(j)} \rangle \right) \\
	=& \sum_{r=1}^{2g} \sum_{*ij \in \tilde{\gamma}_r}  \left(\langle 	\rho_{\delta_r}(f^{\dagger}_{ij,l} - f^{\dagger}_{ij,r}) ,\rho_{\delta_r}(f_j \times \dot{f}_j) \rangle - \langle 	f^{\dagger}_{\delta_r(ij,l)} - f^{\dagger}_{\delta_r(ij,r)}, f_{\delta_r(j)} \times \dot{f}_{\delta_r(j)} \rangle \right)\\
	=& \sum_{r=1}^{2g} \sum_{*ij \in \tilde{\gamma}_r}  \left(\langle 	f^{\dagger}_{\delta_r(ij,l)} - f^{\dagger}_{\delta_r(ij,r)} ,\rho_{\delta_r}(f_j \times \dot{f}_j) \rangle - \langle 	f^{\dagger}_{\delta_r(ij,l)} - f^{\dagger}_{\delta_r(ij,r)}, f_{\delta_r(j)} \times \dot{f}_{\delta_r(j)} \rangle \right)\\
	=& \sum_{r=1}^{2g} \sum_{*ij \in \tilde{\gamma}_r}  \langle 	f^{\dagger}_{\delta_r(ij,l)} - f^{\dagger}_{\delta_r(ij,r)} ,- \eta(\sigma_{\delta_r}) - \langle \eta(\sigma_{\delta_r}), f_{\delta_r(j)} \rangle f_{\delta_r(j)}  \rangle \\
	=& \sum_{r=1}^{2g} \sum_{*ij \in \tilde{\gamma}_r}  \langle 	f^{\dagger}_{\delta_r(ij,l)} - f^{\dagger}_{\delta_r(ij,r)} ,- \eta(\sigma_{\delta_r})  \rangle \\
	=& - \sum_{r=1}^{2g} \tr\left( \sigma_{\delta_r}  \cdot \eta^{-1}(\sum_{*ij \in \tilde{\gamma}_r}  	(f^{\dagger}_{\delta_r(ij,l)} - f^{\dagger}_{\delta_r(ij,r)})) \right) 
\end{align*}
where we used the fact that $f^{\dagger}_{\delta_r(ij,l)} - f^{\dagger}_{\delta_r(ij,r)}$ is orthogonal to $f_{\delta_r(j)}$. Observe that the summation 
\[
\eta^{-1}\left(\sum_{*ij \in \tilde{\gamma}_r}  (	f^{\dagger}_{\delta_r(ij,l)} - f^{\dagger}_{\delta_r(ij,r)})\right)
\]
is simply the evaluation of $\tau$ at some element in $\pi_1(S)$. Comparing it with Equation \eqref{eq:wgexplicit}, we deduce that
	\[
\frac{d}{dt} D_c(f^{(t)})|_{t=0} = \omega_G(\sigma,\tau).
\]
	\end{proof}

We are now ready to prove the main Theorem that at the optimal hyperbolic metric, the discrete harmonic map and the edge weight are induced from a weighted Delaunay decomposition as in Definition \ref{def:wdelaunay}.

\begin{proof}[Proof of Theorem \ref{thm:delaunay}]
	Since the Weil-Petersson symplectic form $\omega_G$ is non-degenerate, the optimal hyperbolic metric is reached if and only if $[\tau]$ is trivial, i.e. $\tau \in B^{1}_{\Ad \rho}(\pi_1(S),so(2,1))$ is a coboundary. Recall that $f^{\dagger}$ is defined up to translation. At the optimal hyperbolic metric, there is a unique dual surface $f^{\dagger}$ such that $\tau \equiv0$ and $f^{\dagger}$ becomes $\rho$-equivariant
	\[
	f^{\dagger}_{\gamma(i)}=\rho_{\gamma} (f^{\dagger}_i).
	\] 
	
	We claim that this $f^{\dagger}$ corresponds to a weighted Delaunay decomposition as in Definition \ref{def:wdelaunay}, with vertex weights $\delta:V \to \mathbb{R}_{>0}$ satisfying for every vertex $i\in V$ and any face $\phi \in F$ containing $i$ 
	\[
	\langle f^{\dagger}_{\phi}, f_i \rangle = -\delta_i
	\] 
	In the rest of the proof, we verify the claim.
	
	By the result of \cite{CdV1991}, a discrete harmonic map is an embedding and defines a geodesic cell decomposition of $(S,h)$. Thus the developing map $f$ also defines a geodesic decomposition of the hyperboloid. Consider $i \in \tilde{V}$, the dual face under $f^{\dagger}$ is spacelike with normal $f_i$. Since the edge weights are furthermore positive, every dual face under $f^{\dagger}$ is convex with non-degenerate edges. By considering the Gauss map $f$, we deduce that $f^{\dagger}$ defines a convex polyhedral surface in $\mathbb{R}^{2,1}$. 
	
	Since the polyhedral surface is equivariant with respect to $\rho$, namely for $\gamma \in \pi_1(S)$
	\[
	f^{\dagger}_{\gamma(\phi)}=\rho_{\gamma} (f^{\dagger}_{\phi})
	\]
	we deduce that it is asymptotic to the null cone \[\{ x \in \mathbb{R}^{2,1}| \langle x,x \rangle =0\}.\] The polyhedral surface being spacelike implies that it must lie within the positive light cone \[
	\{ x\in \mathbb{R}^{2,1} | \langle x,x \rangle <0, x_3>0\}.\]
	
	Thus there exists a function $\tilde{\delta}:\tilde{V} \to \mathbb{R}_{>0}$ such that for any $i\in \tilde{V}$ and any face $\phi \in \tilde{F}$ containing $i$ 
	\[
	\langle f^{\dagger}_{\phi}, f_i \rangle = -\tilde{\delta}_i.
	\] 
	The function $\tilde{\delta}$ satisfies
	\[
	\tilde{\delta}_{\gamma(i)}=\tilde{\delta}_{i}
	\]
	for any $\gamma \in \pi_1(S)$ and thus descends to a well-defined vertex weights $\delta:V \to \mathbb{R}_{>0}$. Since $f^{\dagger}$ is convex, we deduce that $f$ is a weighted Delaunay decomposition with vertex weight $\delta$. Furthermore, by the construction of $f^{\dagger}$, we have for $ij \in E$
	\[
	c_{ij}=	\frac{|| f^{\dagger}_{ij,l} - f^{\dagger}_{ij,r}||}{\ell_{ij}} 
	\]
	and hence $c$ is the edge weight induced from the weighted Delaunay decomposition.
\end{proof}

\begin{remark}
A discrete harmonic map $f$ and its associated dual surface $f^{\dagger}$ should be regarded as the hyperbolic analogue of the Maxwell-Cremona correspondence. In the torus case \cite{Lam2022}, a Euclidean metric is optimal for the energy if and only if both $f$ and $f^{\dagger}$ can project to the same torus. In the hyperbolic case, we have a similar statement that a hyperbolic metric is optimal if and only if both $f$ and $f^{\dagger}$ are $\rho$-equivariant and hence project to the same hyperbolic surface. Then we obtain two geodesic decompositions of a closed hyperbolic surface such that combinatorially they are dual to each other and the corresponding edges are orthogonal to each other.
\end{remark}

\section{Variants of Dirichlet energy}\label{sec:genen}

Similar conclusion can be made generally for Dirichlet energy in the form
\begin{equation}
	\label{eq:genergy}
		\tilde{D}(f) := \sum_{ij } w_{ij}(\ell_{ij})
\end{equation}
where for each $ij\in E$, $w_{ij}:\mathbb{R} \to \mathbb{R}$ is a smooth increasing function of hyperbolic edge lengths $\ell_{ij}$. It would be interesting to explore whether in this general form, there exists a unique discrete harmonic map to a fixed hyperbolic surface and the Dirichlet energy of discrete harmonic maps has a unique minimizer over the Teichm\"{u}ller space. Anyhow with the same techniques, one can show that a discrete harmonic map at an optimal hyperbolic metric yields a weighted Delaunay decomposition.

In terms of the developing map into the hyperboloid with holonomy $\rho$, we have
\[
\tilde{D}({f}) = \sum_{ij} w_{ij}(\cosh^{-1}(-\langle f_i,f_j \rangle)
\]
where the summation is over all edges in a fundamental domain. The realization is discrete harmonic in a fixed hyperbolic surface if for every vertex $i \in V$
\[
\sum_{j} w'_{ij}(\ell_{ij}) \frac{f_i \times f_j}{||f_i \times f_j||} =0.
\]
If an optimal hyperbolic metric is reached, the discrete harmonic map induces a realization of the dual graph $f^{\dagger}:(V^*,E^*) \to \mathbb{R}^{2,1}$ such that 
\begin{align} \label{eq:fdagw}
	f^{\dagger}_{ij,l} - f^{\dagger}_{ij,r}= w'_{ij}(\ell_{ij}) \frac{f_i \times f_j}{||f_i \times f_j||}
\end{align}
and $f^{\dagger}$ is equivariant with respect to $\rho$ in the sense that $f^{\dagger} \circ \gamma = \rho_{\gamma} \circ f^{\dagger}$ for $\gamma \in \pi_1(S)$. The map $f^{\dagger}$ induces a weighted Delaunay decomposition on the hyperbolic surface with \[ ||f^{\dagger}_{ij,l} - f^{\dagger}_{ij,r}|| = w'_{ij}(\ell_{ij}).\]

In the previous sections, we have explored the choice $w_{ij}(x) = \frac{1}{2}c_{ij} x^2$. In fact, there is another interesting choice for the function $w_{ij}$, which is compatible with discrete differential geometry.
\begin{example}
	We consider the Dirichlet energy in the form of Equation \eqref{eq:genergy} with
	\[
	w_{ij}(x) = 2 c_{ij} \sinh^2 \frac{x}{2}
	\]
	for some positive edge weights $c:E \to \mathbb{R}_{>0}$. At an optimal hyperbolic metric, the discrete harmonic map $f$ induces a realization of the dual graph $f^{\dagger}:(V^*,E^*) \to \mathbb{R}^{2,1}$ such that by Equation \eqref{eq:fdagw},
	\[
	c_{ij} = \frac{||f^{\dagger}_{ij,l} - f^{\dagger}_{ij,r}||}{\sinh \ell_{ij}}= \frac{||f^{\dagger}_{ij,l} - f^{\dagger}_{ij,r}||}{||f_i \times f_j||}.
	\]
	In case $f^{\dagger}$ yields a Delaunay decomposition (i.e. trivial weights at vertices), then the weight $c$ coincides exactly with the edge weights in the discrete hyperbolic Laplacian \cite{Lam2024}, i.e.
	\[
	  c_{ij}=\frac{\cot( \frac{\pi - \alpha_{jk}^i - \alpha_{ki}^j + \alpha_{ij}^k}{2}) + \cot( \frac{\pi - \alpha_{lj}^i - \alpha_{il}^j + \alpha_{ji}^l}{2})}{2  \cosh^2 \frac{\ell_{ij}}{2}}
	\]
	which is related to deformation of circle patterns. Furthermore, the energy of the discrete harmonic map $f$ becomes
    \begin{align*}
    	\tilde{D}({f})=  	\sum_{ij} ||f^{\dagger}_{ij,l} - f^{\dagger}_{ij,r}|| \tanh \frac{\lambda_{ij}}{2}
    	\end{align*}
    which is the integrated mean curvature of the space-like polyhedra surface $f^{\dagger}$ (See \cite{LamYashi2017}). There is also a similar interpretation for the choice $w_{ij}(x) = \frac{1}{2}c_{ij} x^2$ (See \cite{Roman2023}).
\end{example}
	
	\section*{Acknowledgment}
	
	The author would like to thank Toru Kajigaya for fruitful discussions.

	\bibliographystyle{amsplain}
\bibliography{torus}

\end{document}